\documentclass[12pt]{amsart}

\usepackage{fullpage}
\usepackage{graphicx}
\usepackage{color}

\definecolor{orange}{rgb}{1,0.5,0}
\definecolor{purple}{rgb}{0.5,0,1}
\definecolor{darkgreen}{rgb}{0,0.5,0}
\definecolor{darkcyan}{rgb}{0,0.5,0.5}

\newtheorem{thm}{Theorem}
\newtheorem{prop}[thm]{Proposition}
\newtheorem{lemma}[thm]{Lemma}
\newtheorem{cor}[thm]{Corollary}
\theoremstyle{definition}
\newtheorem{definition}[thm]{Definition}

\newcommand{\Aa}{A_1} 
\newcommand{\Ab}{A_2} 
\newcommand{\Ac}{A_3} 
\newcommand{\Ad}{A_4} 
\newcommand{\Ae}{A_5} 
\newcommand{\Af}{A_6} 
\newcommand{\Ai}{A_i}
\newcommand{\Aj}{A_j}
\newcommand{\Ba}{B_1}
\newcommand{\Bb}{B_2} 
\newcommand{\Bc}{B_3} 
\newcommand{\Bd}{B_4} 
\newcommand{\Be}{B_5} 
\newcommand{\Bf}{B_6} 
\newcommand{\Bg}{B_7}
\newcommand{\Bi}{B_i}
\newcommand{\Bk}{B_k} 

\newcommand{\dod}[3]{\raisebox{-0.5cm}{\includegraphics[scale = 0.60]{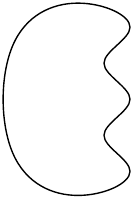}}\;\begin{matrix} #1 \\ #2 \\ #3 \end{matrix}}
\newcommand{\dood}[4]{\raisebox{-0.7cm}{\includegraphics[scale = 0.40]{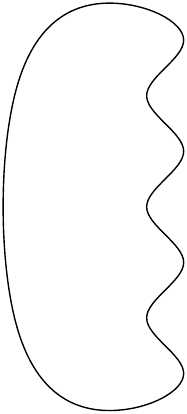}}\;\begin{matrix} #1 \\ #2 \\ #3 \\ #4 \end{matrix}}

\newcommand{\adj}{\mathop{\mathrm{adj}}}

\title{A four-vertex, quadratic, spanning forest polynomial identity}

\author{Aleksandar Vlasev}
\author{Karen Yeats}
\thanks{The authors thank NSERC for its support}

\begin{document}

\begin{abstract}
The classical Dodgson identity can be interpreted as a quadratic identity of spanning forest polynomials, where the spanning forests used in each polynomial are defined by how three marked vertices are divided among the component trees.
We prove an analogous result with four marked vertices.
\end{abstract}

\maketitle


\section{Introduction}

Let $G$ be a graph with $m$ vertices, $n$ edges and let the $i^{\text{th}}$ edge be assigned a variable $\alpha_i$. Then we define the graph polynomial of $G$ as
\[
    \Psi_G = \sum_{T\subseteq G}\prod_{e\notin T}\alpha_e
\]
where the sum runs over the spanning trees $T$ of $G$. The reason why we pick edges not in the trees is that this form arises naturally in quantum field theory, see for example \cite{BW, Brphi4, Kcore}. We can also obtain this polynomial via the matrix-tree theorem. Let $A$ be the $n \times n$ diagonal matrix with the variables $\alpha_i$. Orient the edges in the graph and let $E$ be the signed $m \times n$ incidence matrix for this orientation. Let $\widehat{E}$ be the matrix $E$ with any row removed. Define the $m+n$ by $m+n$ block matrix
\[
    M_G = \begin{bmatrix}A & \widehat{E}^T \\ -\widehat{E} & 0\end{bmatrix}.
\]
Then the matrix-tree theorem states that
\[
    \Psi_G = \det(M_G).
\]
To put this in the usual form of the matrix-tree theorem, note that $A$ is invertible, so we can calculate the determinant using the Schur complement; in this case, 
\begin{align*}
  \det(M) & = \alpha_1\cdots\alpha_n \det(0-(-\widehat{E}A^{-1}\widehat{E}^T)) \\
  & =\alpha_1\cdots\alpha_n \det(\widehat{E}A^{-1}\widehat{E}^T)
\end{align*}where $\widehat{E}A^{-1}\widehat{E}^T$ is the graph Laplacian with a row and column removed and with inverted variables.  See also Proposition 21 of \cite{Brbig}. 

There are two important ways to generalize $\Psi_G$ -- one via the polynomials and one via the matrix determinant. Let $P = P_1\cup\dots\cup P_k$ be a set partition of a set of vertices in $G$. Then define the spanning forest polynomial for $G$ and $P$ as
\[
    \Phi_G^P = \sum_{F \subseteq G}\prod_{e\notin F} \alpha_e
\]
where the sum runs over spanning forests $F$ of $G$ composed of tree components $T_1,\dots,T_k$ where the vertices $P_i$ are in tree $T_i$. Alternatively, let $I,J,K$ be sets of indices with $|I|=|J|$. Define the Dodgson polynomial $\Psi_{G,K}^{I,J}$ as
\[
    \Psi_{G,K}^{I,J} = \det(M_G(I,J))_K
\]
where $M_G(I,J)$ is the submatrix obtained by removing the rows indexed by $I$ and the columns indexed by $J$ from $M_G$, and the subscript $K$ indicates that we are setting the variables $\alpha$ indexed by $K$ to 0. These two generalizations are related -- every Dodgson polynomial can be expressed as a sum of signed spanning forest polynomials (see \cite{BrY}). Thus we can use determinant identities to derive identities for spanning forest polynomials. For any square matrix $M$, we have the classical Dodgson identity
\[
  \det(M(12,12))\det(M)= \det(M(1,1))\det(M(2,2)) - \det(M(1,2))\det(M(2,1))
\]
which was popularized by Dodgson through his condensation algorithm \cite{dodgson}. Let $G$ be a graph of the form
\[
   \includegraphics{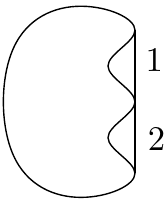}
\]
with two edges labelled $1$ and $2$, connecting three vertices $v_1,v_2$ and $v_3$ from top to bottom. The Dodgson identity gives the spanning forest polynomial identity (see section \ref{sec Dodgson id})
\begin{equation}\label{3 and 1 id}
   \dod{}{}{}\dod{1}{2}{3} = \dod{1}{1}{2} \dod{1}{2}{1} +\dod{1}{1}{2} \dod{1}{2}{2} +\dod{1}{2}{1} \dod{1}{2}{2}
\end{equation}
where for example, the graph with labels $1,1,2$ on the vertices $v_1,v_2,v_3$ represents $\Phi_G^P$ with $P =\{v_1,v_2\}\cup\{v_3\}$.

This result can be interpreted as saying that if we transfer an extra edge from the left hand factor of the left hand side to the right hand factor of the left hand side, thus cutting a spanning tree into two in the left hand factor and joining two of the three trees together in the right hand factor, then we get all pairs of spanning forests with exactly two trees.  However, it is subtle to see that the counting matches on both sides, and seems to require chains of edges to be transferred, along the lines of the the combinatorial proof of the Dodgson identity due to Zeilberger \cite{Zeil}.

Equation \eqref{3 and 1 id} and its combinatorial interpretation prompted us to investigate spanning forest polynomial identities of the form
\[
  \dood{}{}{}{} \dood{1}{2}{3}{4} = \dood{a_1}{a_2}{a_3}{a_4} \dood{b_1}{b_2}{b_3}{b_4} + \dood{c_1}{c_2}{c_3}{c_4} \dood{d_1}{d_2}{d_3}{d_4} +\cdots + \dood{e_1}{e_2}{e_3}{e_4} \dood{f_1}{f_2}{f_3}{f_4} 
\]
Our work resulted in such an identity (Theorem \ref{main thm}) which is proved in this paper. For this result we cannot simply interpret a classical determinantal identity; the Jacobi identity on $M$ (see Corollary \ref{jacobi}) naturally gives a cubic identity for such spanning forest polynomials, while the usual Dodgson identities on submatrices of $M$ can only relate spanning forest polynomials whose degrees differ by at most 2.  Rather, we need to combine classical identities in nontrivial ways.

The paper is organized as follows: In section \ref{sec graph polys} we will set up our definitions.  In section \ref{sec Dodgson id} we will define spanning forest polynomials and give their relation to the minors of $M$.  The main result itself is presented and proved in section \ref{sec main}. Finally, in section \ref{sec concl} we conclude with a discussion of the main result, its combinatorial interpretations, and possible extensions.


\section{Graph polynomials}\label{sec graph polys}

\begin{definition} Let $G$ be a graph and let $M_G$ be a matrix built as in the previous section. Then we define
  \[
  \Psi_G = \det(M_G)
  \]
\end{definition}
By the matrix-tree theorem, $\Psi_G$ is independent of the choice of $M_G$.  We will call $\Psi_G$ the graph polynomial or Kirchhoff polynomial of $G$. We fix a choice of $M = M_G$ for $G$.
\begin{definition}
  Let $I$, $J$, and $K$ be subsets of the edges of $G$ with $|I|=|J|$.  Let $M(I,J)_K$ be the matrix obtained from $M$ by removing the rows indexed by edges of $I$, the columns indexed by edges of $J$, and setting $\alpha_i=0$ for all $i\in K$.  Then we define the Dodgson polynomials
  \[
  \Psi^{I,J}_{G,K} = \det M(I,J)_K
  \]
 When $G$ is clear it will be suppressed from the notation. Also, if $K=\emptyset$ we may suppress it from the notation.
\end{definition} 
Up to sign these polynomials are independent of the choice of $M$ (see \cite{Brbig}). By definition it is evident that $\Psi^{\emptyset, \emptyset}_{G,\emptyset} = \Psi_G$.  Note that if any element of $K$ appears in $I$ or $J$ then it does not appear in $M(I,J)$, so setting it to zero has no effect.

Contraction and deletion of edges is natural at the level of Dodgson polynomials.
\begin{prop} Let $G$ be a graph and let $e_i$ denote the $i$-th edge in $G$. Then
  \begin{align*}
    \Psi_G^{i,i} & = \Psi_{G\backslash e_i} \\
    \Psi_{G,i} & = \Psi_{G/e_i}
  \end{align*}
\end{prop}
\begin{proof} The first identity follows immediately from the matrix definition of $\Psi$ and the second from the sum of spanning trees definition.
\end{proof}

The all-minors matrix-tree theorem \cite{Chai} tells us that the monomials of any $\Psi^{I,J}_{G,K}$ result from spanning forests of $G$. For our purposes it is most useful to organize these spanning forests with the following spanning forest polynomials.

\begin{definition}\label{spanning forest def}
  Let $P= P_1\cup P_2\cup \cdots\cup P_k$ be a set partition of a subset of the vertices of $G$. Then we define
  \[
  \Phi^P_G = \sum_F\prod_{e\not\in F}\alpha_e 
  \]
  where the sum runs over spanning forests $F=T_1\cup T_2 \cup \cdots \cup T_k$ with $k$ component trees so that the vertices of $P_i$ are in tree $T_i$.  We note that we are allowing trees consisting of a single vertex.
\end{definition}

The relation between Dodgson polynomials and spanning forest polynomials is given by the following proposition.
\begin{prop}
\begin{equation}\label{Psi to Phi}
  \Psi^{I,J}_{G,K} = \sum_P \pm \Phi^P_{G \backslash (I\cup J\cup K)}
\end{equation}
where the sum runs over set partitions $P$ of the end points of edges of $I$, $J$, and $K$ with the property that the forests corresponding to each set partition become trees in both \[G \backslash I / (J\cup K) \quad \text{and} \quad G \backslash J / (I\cup K)\] 
\end{prop}

\begin{proof}  For the full details, see Proposition 12 of \cite{BrY}.  To sketch the argument, equation \eqref{Psi to Phi} is a direct consequence of two facts. Let $\widehat{E}[S]$ be the submatrix of $\widehat{E}$ consisting of columns indexed by $S$. First, the coefficient of a given monomial $m$ in $\Psi^{I,J}_K$ is 
\[
  \det\begin{bmatrix} 0 & \widehat{E}[J\cup K\cup F]^T \\ -\widehat{E}[I\cup K\cup F] & 0 \end{bmatrix} 
\]
where $F$ is the forest corresponding to $m$  (that is the edges of $G\backslash (I\cup J\cup K)$ which do not contribute to $m$). This fact follows directly from the form of $M$. Second, a square matrix formed of columns of $\widehat{E}$ has determinant $\pm 1$ if the edges corresponding to those columns are a spanning tree of $G$, and has determinant $0$ otherwise.  This fact is the matrix-tree theorem in its most stripped down form, see for example \cite{Brbig} Lemma 20.  
\end{proof}



\section{The classical Dodgson identity}\label{sec Dodgson id}

In this section we interpret the classical Dodgson identity in terms of spanning forest polynomials. Consider the graph $G$
\[
\includegraphics{tie}
\]
Apply the Dodgson determinant identity to the matrix $M$ for $G$
\[
     \det(M(1,1))\det(M(2,2)) - \det(M(1,2))\det(M(2,1)) = \det(M)\det(M(12,12))
\]
Interpreting this in terms of Dodgson polynomials gives
\[
    \Psi_G^{1,1} \Psi_G^{2,2} - \Psi_G^{1,2}\Psi_G^{1,2} = \Psi_G \Psi_G^{12,12}
\]
and after setting the variables for edges 1 and 2 to 0 we obtain
\[
\Psi^{1,1}_{G,2} \Psi^{2,2}_{G,1} - (\Psi_G^{1,2})^2 = \Psi_G \Psi_G^{12,12}.
\]
For a generalization, see Corollary \ref{second cor}. Using the deletion-contraction relations we obtain
\[
   \Psi_{G\backslash e_1/e_2} \Psi_{G\backslash e_2/e_1} - (\Psi_G^{1,2})^2 = \Psi_G \Psi_{G\backslash \{e_1,e_2\}}
\]
and converting to spanning forest polynomials we find that
\[
\left(\Phi_H^{\{a,c\}, \{b\}}+\Phi_H^{\{a\}, \{b,c\}}\right)\left(\Phi_H^{\{a,b\}, \{c\}}+\Phi_H^{\{a,c\},\{b\}}\right) -\left (\pm\Phi_H^{\{a,c\},\{b\}}\right)^2 = \Phi_H^{\{a,b,c\}}\Phi_H^{\{a\},\{b\},\{c\}}
\]
where $H$ is the graph with edges $1$ and $2$ removed. Rearranging and cancelling the squared term we find that
\[
     \Phi_H^{\{a,b,c\}}\Phi_H^{\{a\},\{b\},\{c\}} = \Phi_H^{\{a,b\}, \{c\}}\Phi_H^{\{a,c\}, \{b\}}+\Phi_H^{\{a,b\}, \{c\}}\Phi_H^{\{a\}, \{b,c\}}+\Phi_H^{\{a,c\},\{b\}}\Phi_H^{\{a\}, \{b,c\}}
\]
which is just equation \eqref{3 and 1 id} written in the spanning forest polynomial notation.  See Proposition 22 in \cite{BrY} for more details.


\section{The main result}\label{sec main}
In section \ref{sec Dodgson id} we gave the spanning forest polynomial version of the Dodgson identity.  The main result of this paper is an analogous spanning forest polynomial identity for 4 marked vertices.  Let us specialize our notation to this situation.

\begin{definition}\label{A and B}
Let $v_1$, $v_2$, $v_3$, and $v_4$ be four distinct vertices of a graph $G$.  We will write $(c_1,c_2,c_3,c_4)$ with $c_i \in \{1,2,3,4,-\}$ to denote the spanning forest polynomial of the graph $G$ defined by the partition of $\{v_i: c_i \neq -\}$ with one part for each distinct integer $\ell$ in $(c_1,c_2,c_3,c_4)$ defined by $\{v_i: c_i=\ell\}$, and no other parts. Particularly useful are the following special cases
\begin{align*}
   \Aa &= (1,1,2,3), \quad \Ab = (1,2,1,3), \quad \Ac = (1,2,2,3),\\
   \Ad &= (1,2,3,1), \quad \Ae = (1,2,3,2), \quad \Af = (1,2,3,3)
\end{align*}
of 3 marked vertices each, the following cases
\begin{align*}
   \Ba = (1,1,1,2), \quad &\Bb = (1,1,2,1), \quad \Bc = (1,2,1,1),\\
   \Bd = (1,2,2,2), \quad &\Be = (1,1,2,2), \quad \Bf = (1,2,1,2),\\
                    \quad &\Bg = (1,2,2,1)
\end{align*}
of 3 marked vertices each, and finally let
\[
  P = (1,1,1,1).
\]
\end{definition}
The $\Ai$ and $\Bi$ are the different ways in which we can partition four vertices in 3 and 2 sets respectively.  $P$ is just $\Psi_G$ for this $G$ with four marked vertices.

\begin{thm}\label{main thm}Let $G$ be a graph with four marked vertices. Then
\begin{equation}
\begin{aligned}\label{emp formula}
(1,1,1,1)(1,2,3,4) &= \quad(1 - x_ 1 - x_ 2) \Ad  \Ba + x_ 7 \Ab  \Bd+ (1 - x_ 3 - x_ 2) \Ae  \Ba \\
&\quad+ (1 - x_ 1 - x_ 4) \Af  \Ba +  x_ 2 \Ab  \Bb + ( x_ 3 + x_ 2- x_ 5 ) \Ac  \Bb \\
&\quad+ (1 - x_ 1 - x_ 6) \Af  \Bb +  x_ 1 \Aa  \Bc + ( x_ 1 - x_ 7 + x_ 4) \Ac  \Bc \\
&\quad+ ( x_ 1 - x_ 8 + x_ 6) \Ae  \Bc + x_ 5 \Aa  \Bd + ( x_ 1 - x_ 5 + x_ 4) \Ac  \Be \\
&\quad+ ( x_ 1 - x_ 5 + x_ 6) \Ae \Be +  x_ 3 \Aa  \Bf + (   x_ 3 + x_ 2 - x_ 7) \Ac  \Bf \\
&\quad+ (1 - x_ 1 - x_ 2 + x_ 8 -   x_ 6) \Ad  \Bf + (   x_ 2 + x_ 7 - x_ 4) \Ab  \Bg  \\
&\quad+(1 - x_ 1 - x_ 7 + x_ 8 - x_ 6) \Af  \Bf + (  x_ 1 + x_ 5 - x_ 3) \Aa  \Bg \\
&\quad+ (1 + x_ 5 - x_ 3 - x_ 2 -  x_ 8) \Ae  \Bg \\
&\quad+ (1 - x_ 1 + x_ 7 - x_ 4 - x_ 8) \Af  \Bg \\
&\quad + x_ 8 \Ad  \Bd +  x_ 4 \Ab  \Be  +  x_ 6 \Ad  \Be \\
\end{aligned}
\end{equation}
for any $x_1, \ldots , x_8$.
\end{thm}

This is the generalization of the classical Dodgson identity phrased in terms of spanning forest polynomials.  It is possible to give a graphical representation of this identity like in equation \eqref{3 and 1 id} but it would take too much space.
\begin{proof}[Outline of proof]
Here we will describe the structure of the proof and the necessary calculations will be carried out in the results which follow this outline.

Let $E(x_1, x_2, x_3, x_4, x_5, x_6, x_7, x_8)$ be the right hand side of \eqref{emp formula}.  We first check that $E$ does not depend on the values of the $x_i$ by checking that the coefficient of each $x_i$ in $E$ is zero (Proposition \ref{free vars}).  Now we are free to use any choice of $x_i$ which is algebraically convenient.

Next, using the classical Jacobi identity on an auxiliary graph with three extra edges we obtain an expression for
\[
  (1,2,3,4)^2(1,1,1,1) 
\]
which is a linear combination of products of the form $\Ai\Aj$ (Lemma \ref{apply Jacobi}).

Then we calculate each $P\Ai$ as a linear combination of products of the form $\Aj\Bk$ (Lemma \ref{PAi}).
Using this calculation we obtain an expression for $\big((1,2,3,4)(1,1,1,1)\big)^2$, which we can then check is the same as $E(0,0,0,0,0,0,0,0)E(0,1,0,1,1,1,1,1)$ (Lemma \ref{main calc}).

The proof of the theorem concludes by checking the sign.
\end{proof}

Let us pause here for a brief word on the role of the $x_i$.  The Dodgson identities give a number of quadratic identities between the $\Ai$ and $\Bi$.  Consequently there cannot be a unique way to write $(1,1,1,1)(1,2,3,4)$ as a linear combination of products $\Ai\Bk$.  The $x_i$'s describe this nonuniqueness.  We can specialize to get more manageable equations, for example setting all $x_i=0$ and
collecting terms gives
\[
  (1,2,3,4)(1,1,1,1) = (1,2,3,1)(1,-,1,2) + (1,2,3,2)(-,1,1,2) + (1,2,3,3)(-,-,1,2)
\]
but no such specialization is canonical, so we gave the general equation in Theorem \ref{main thm}.


\subsection{Preliminary results}

Now we can proceed with the lemmas.
We will need a particular form of the Jacobi determinantal identity and some further Dodgson identities which follow from it.

Let $M$ be an $n\times n$ matrix. 
Let $I$ and $J$ be subsets of $\{1,2,\ldots, n\}$.
Let $M(I,J)$ be the matrix obtained from $M$ by removing rows $I$ and columns $J$. Similarly let $M[I,J]$ be the matrix where we only keep rows $I$ and columns $J$. Finally we let
\[
   s(I,J) = \sum_{x \in I}x + \sum_{x \in J}x.
\]

\begin{thm}\label{thm jacobi}Let $M$ be a nonsingular $n \times n$ matrix and let $I$ and $J$ be two sets in $\{1,2,\ldots,n\}$ with $|I|=|J|=t$. Let $A = \adj M$ and define the matrix $B$ by $b_{ij} = \det(M(i,j))$.  Then 
\[
   \det (B[I,J]) =(\det M)^{t-1} \det (M(I,J))
\]
\end{thm}
\begin{proof}
To remain self contained we will give a proof following the idea of the proof of Lemma 28 of \cite{Brbig}.
Let $I_n$ be the $n\times n$ identity matrix.  Then
\[
 A M = I_n (\det M)
\]
Take determinants to get
\[
  \det(A) = \det M^{n-1}.
\]
Now if the $k$-th element of $I$ is $i_k$ and the $k$-th element of $J$ is $j_k$ let $C$ be $M$ with the $j_k$ column replaced by $\mathbf{e}_{i_k}$, where $\mathbf{e}_i$ is the $i$-th standard basis element of $\mathbb{R}^n$. Then multiplying out column by column we get that $AC$ is the matrix $D$ whose $j$-th column is
\[
\begin{cases}
  (\det M)\mathbf{e}_j & \text{if $j$ is not in $J$} \\
  A\mathbf{e}_{i_k} & \text{if $j$ is $j_k$ in $J$}
  \end{cases}
\]
Now notice that
\[
\det C = (-1)^{s(I,J)} \det(M(I,J))
\]
and
\begin{align*}
\det D & = (\det M)^{n-t} \det(A[J,I]) \\
& = (\det M)^{n-t}\det(B[I,J]) (-1)^{s(I,J)}.
\end{align*}
The second equality holds since $A[J,I]$ can be converted to $B[I,J]$ by multiplying each row and each column which had an odd index in $M$ by $-1$ and then taking a transpose; on determinants this changes the sign $s(I,J)$ times.

Finally, taking the determinant of $AC=D$, using the above calculations and dividing by $(\det M)^{n-t}$ gives us the result.\end{proof}

This formula can readily be translated into the Dodgson polynomials language.

\begin{cor}\label{jacobi} Let $G$ be a graph and $M$ be its associated matrix. Let $I$, $J$ and $E$ be subsets of the edges, such that $|I|=|J|=k$. Then the $k$-level Dodgson identity is
\[
   \det\left(\Psi_{G,E}^{I_i,J_j}\right)_{1\leq i,j\leq k} = \Psi_{G,E}^{I,J}(\Psi_{G,E}^{})^{k-1}
\]
where $I=\{I_1, \ldots, I_k\}$ and $J=\{J_1, \ldots, J_k\}$.
\end{cor}
\begin{proof}Use Theorem \ref{thm jacobi}. By definition $\det M = \Psi_G$ and $\det M(I,J) = \Psi_G^{I,J}$. Now $B[I,J]_{ij} = \det(M(I_i,J_j)) = \Psi^{I_i,J_j}$. Finally we set $\alpha_e = 0$ for $e\in E$.

\end{proof}
Careful book-keeping and application of the above identity yield the following corollary.

\begin{cor}\label{second cor}Let $M$ be an associated matrix for the graph $G$. Let $E$, $I$, $J$, $A$ and $B$ be ordered sets indexing edges in $G$, such that $|A\cap I|=|B \cap J| = 0$, $|I|=|J|=k$ and $|A|=|B|=l$. Then the modified $k$-level Dodgson identity is
\begin{equation}
   \det\left(\Psi_{G,E}^{A\cup I_i,B\cup J_j}\right)_{1\leq i,j\leq k} = \Psi_{G,E}^{A\cup I,B\cup J}\left(\Psi_{G,E}^{A,B}\right)^{k-1}
   \label{klevelmod}
\end{equation}
where $I=\{I_1, \ldots, I_k\}$ and $J=\{J_1, \ldots, J_k\}$.
\end{cor}

Note that when $k=2$ this gives the classical Dodgson identity.

We will use the following rearrangement of the $k=2$ case.
\begin{prop}[Brown, \cite{Brbig}]
  Let $I$ and $J$ be subsets of edges of $G$ with $|J|= |I|+1$.  Let $a$, $b$, $x$ be edges indices with $a\not\in I$, $b,x\not\in I\cup J$, and $x<a<b$.  Let $S=I\cup J \cup \{a,b,x\}$.  Then
  \begin{equation}\label{Dodgson c}
    \Psi^{Ia, J}_S\Psi^{Ibx,Jx}_S - \Psi^{Iax,Jx}_S\Psi^{Ib,J}_S = \Psi^{Ix,J}_S\Psi^{Iab,Jx}_S
  \end{equation}
\end{prop}

\begin{proof}
  This is equation 23 from \cite{Brbig}; the proof proceeds by applying the $k=2$ case of \eqref{klevelmod} three times and rearranging.
\end{proof}

We only need the signs relating Dodgson polynomials to spanning forest polynomials in two cases, given in the next lemma.  The general formula is found in Proposition 16 of \cite{BrY}, but we give here a self contained proof of the cases we need.
\begin{lemma}\label{Lemma signs}
  Fix an order and orientation of the edges of a graph $G$.  Suppose edges $1$, $2$, and $3$ have a common vertex $v$.  Let $w_1$, $w_2$, and $w_3$ be distinct and be the other end points of $1$, $2$, and $3$, and let 
  \[
  \epsilon(i,j) = \begin{cases} 1 & \text{if $i$ and $j$ are both oriented into $v$ or both oriented out of $v$} \\
    -1 & \text{otherwise}
    \end{cases}
  \]
  for $i \neq j \in \{1,2,3\}$.
Then
  \[
    \Psi^{1,2} = \epsilon(1,2)\Phi^{\{v\}, \{w_1,w_2\}} 
  \]  
  and
  \[
    \Psi_k^{i,j} = \epsilon(i,j)(-1)^{i-j+1}\Phi^{\{v\}, \{w_i,w_j\}, \{w_k\}}
  \]
  where $\{i,j,k\} =  \{1,2,3\}$ in some order.
\end{lemma} 

\begin{proof}
  The first statement of the lemma follows from the second with $k=3$ applied to the graph $G$ with a new vertex $w_3$ added and a new edge $3$ from $v$ to $w_3$. Consider the second statement.  Let $x$ be the vertex which was removed when forming $M$.  We choose it to be disjoint from $\{v, w_i, w_j\}$.

  Note that $\{v\},\{w_i,w_j\},\{w_k\}$ is the only set partition compatible with $\Psi_k^{i,j}$.  From the observations preceding this lemma, if $\Psi^{i,j}_k = 0$ then there are no common spanning trees of $G\backslash i/\{j,k\}$ and $G\backslash j/\{i,k\}$ and so in particular there are no terms in $\Phi^{\{v\}, \{w_i,w_j\}, \{w_k\}}$.  Thus 
\[
\Psi^{i,j}_k = 0 \quad \Leftrightarrow \quad \Phi^{\{v\}, \{w_i,w_j\}, \{w_k\}}=0.
\]

  By \eqref{Psi to Phi} we know that $\Psi^{i,j}_k = f\Phi^{\{v\}, \{w_i,w_j\}, \{w_k\}}$ for some $f\in\{-1,1\}$, so it suffices to consider one term of $\Psi^{i,j}$. Pick a term $t$ where the tree out of $w_i$ and $w_j$ intersects $x$. Let $F$ be the forest corresponding to $t$.  The sign of $t$ in $\Psi^{i,j}$ is $\det N$ where
\[
N = \begin{bmatrix} 0 & \widehat{E}[\{i,k\}\cup F]^T \\ -\widehat{E}[\{j,k\}\cup F] & 0 \end{bmatrix}
\]

Let $B = \widehat{E}[\{k\}\cup F]$. Then $\widehat{E}[\{i,k\}\cup F]$ and $\widehat{E}[\{j,k\}\cup F]$ are formed by inserting the $i$th and $j$th columns respectively of $\widehat{E}$ into $B$.  If $\{i,j\} = \{1,2\}$ the insertions are both made in the first column.  Let $i'$ be the index of the inserted column $i$ and $j'$ the index of the inserted column in $j$.  Thus if $\{i,j\} = \{1,2\}$ then $i'=j'=1$; if $\{i,j\} = \{1,3\}$ then $\{i', j'\} = \{1,2\}$; and  if $\{i,j\} = \{2,3\}$ then $i'=j'=2$.

Consider $B$ with the row corresponding to $v$ removed.  This is the same as the columns corresponding to edges of $\{k\}\cup F$ in the incidence matrix of the graph with $v$ and $x$ identified.  This has determinant $\pm 1$ since $\{k\}\cup F$ was chosen to be a tree in this graph.  Likewise, removing the row corresponding to $w_1$ or $w_2$ we get a zero determinant since $\{k\}\cup F$ is not a tree in the graph with $w_1$ or $w_2$ identified with $x$.  

Thus if we expand $\det \widehat{E}[\{i,k\}\cup F]$ down the inserted column, only the cofactor coming from row $v$ is retained, and likewise for $\widehat{E}[\{j,k\}\cup F]$.  Thus
\begin{align*}
\det N & = \det(\widehat{E}[\{i,k\}\cup F])\det(\widehat{E}[\{j,k\}\cup F]) \\
& = e_{i,\ell} e_{j,\ell} (-1)^{i'+j'+2\ell} \det(\widehat{B})^2 \\
& = \epsilon(i,j)(-1)^{i-j+1}
\end{align*}
where $\ell$ is the index of row $v$,  $\widehat{B}$ is $B$ with row $v$ removed and $e_{r,s}$ is the $(r,s)$ entry of $\widehat{E}$.
\end{proof}


\subsection{Results for the main argument}

Here is a catalogue of the instances of the Dodgson identity which we will need in the main argument, written in terms of the $\Ai$ and $\Bi$ from Definition \ref{A and B}.
\begin{lemma}\label{Dodgsons in AB}
  \begin{align}
    \Aa(\Bc+\Bg) + \Ab(\Bg-\Be) - \Ad(\Ba+\Be) & = 0 \label{partx1} \\
    \Aa(\Bd+\Bg) + \Ae(\Bg-\Be) - \Ac(\Bb+\Be) & = 0 \label{x5}\\
    \Ab(\Bb+\Bg) + \Aa(\Bg-\Bf) - \Ad(\Ba+\Bf) & = 0 \label{partx2} \\
    \Ab(\Bd+\Bg) + \Af(\Bg-\Bf) - \Ac(\Bc+\Bf) & = 0 \label{x7}\\
    \Ac(\Bb+\Bf) + \Aa(\Bf-\Bg) - \Ae(\Ba+\Bg) & = 0 \label{x3}\\
    \Ac(\Bc+\Be) + \Ab(\Be-\Bg) - \Af(\Ba+\Bg) & = 0 \label{x4}\\
    \Ad(\Bd+\Be) + \Ae(\Be-\Bg) - \Af(\Bb+\Bg) & = 0 \label{formain} \\
    \Ad(\Bd+\Bf) + \Af(\Bf-\Bg) - \Ae(\Bc+\Bg) & = 0 \label{x8}\\
    \Ae(\Bc+\Be) + \Ad(\Be-\Bf) - \Af(\Bb+\Bf) & = 0 \label{x6}
  \end{align}
\end{lemma}

\begin{proof}
The equations differ only by permuting the four marked vertices, so it suffices to prove \eqref{x5}.  Consider the graph
\[
\includegraphics{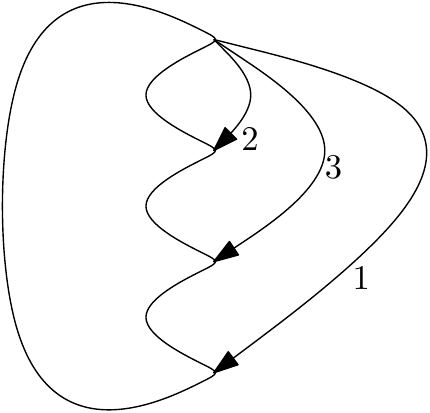}
\]
We use identity \eqref{Dodgson c} with $x=1$, $a=2$, $b=3$, $I = \emptyset$ and $J=\{2\}$, and by Lemma \ref{Lemma signs} we obtain
\[
(1,-,2,3)(1,2,2,-) - (1,-,2,-)(1,2,2,3) = (1,2,3,2)(1,-,2,2).
\]
For the sign of $(1,-,2,2)$ note that the cutting happens first so that edges $1$ and $3$ become adjacent columns in the cut matrix.
Expanding, $(1,-,2,3) = \Aa + \Ac + \Ae$, $(1,2,2,-) = \Bd+\Bg$, $(1,-,2,-) = \Bb + \Bd + \Be+\Bg$, and $(1,-,2,2)= \Bd + \Be$. We substitute these in and rearranging gives us equation \eqref{x5}.
\end{proof}

\begin{prop}\label{free vars}
  All the free variables in \eqref{emp formula} are explained by Dodgson identities.
\end{prop}

\begin{proof}
  The coefficient of $x_3$ in equation \eqref{emp formula} is the right hand side of equation $\eqref{x3}$, and thus is $0$.  Similarly the coefficients of $x_4$, $x_5$, $x_6$, $x_7$, and $x_8$ are zero by \eqref{x4}, \eqref{x5}, \eqref{x6}, \eqref{x7}, and \eqref{x8} respectively.  The coefficient of $x_2$ is in a different form, but is also zero as it is the sum of the right hand sides of \eqref{partx2} and \eqref{x3}.  Finally, the coefficient of $x_1$ is the sum of the right hand sides of \eqref{x6}, \eqref{x4}, and \eqref{partx1} and so is zero.
\end{proof}

\begin{lemma}\label{apply Jacobi}
\[
  (1,1,1,1)(1,2,3,4)^2 = \det 
  \begin{pmatrix}
   \Aa + \Ac + \Ae & -\Ac & -\Ae \\
   -\Ac & \Ab + \Ac + \Af& -\Af\\
   -\Ae & -\Af & \Ad + \Ae + \Af
  \end{pmatrix}
\]
\end{lemma}

\begin{proof}
Let $H$ be $G$ with three new edges $1$, $2$ and $3$ connecting vertex $v_1$ with the other 3 marked vertices. By Corollary \ref{jacobi} with $k = 3$ and $I = J = E = \{1,2,3\}$ we have
\begin{equation} \left(\Psi_{H,123}\right)^2 \Psi_{H}^{123,123} =\det 
  \begin{pmatrix}
   \Psi_{H,23}^{1,1} &\Psi_{H,3}^{1,2} &\Psi_{H,2}^{1,3} \\
   \Psi_{H,3}^{1,2} &\Psi_{H,13}^{2,2} &\Psi_{H,1}^{2,3} \\
   \Psi_{H,2}^{1,3} &\Psi_{H,1}^{2,3} &\Psi_{H,12}^{3,3}
  \end{pmatrix}
\end{equation}
where $\Psi_{H}^{123,123}$ is the graph polynomial of $G$ with the edges 1,2 and 3 removed, namely $\Psi_{H}^{123,123}=P=(1,1,1,1)$; $\Psi_{H,123}$ is the spanning forest polynomial of $G$ where each of the four vertices is in a separate tree, namely $\Psi_{H,123}=(1,2,3,4)$. 

The Dodgson polynomials on the main diagonal are just spanning forest polynomials of $G$ where one of the edges is removed and the other two contracted. By inspection, these are precisely the terms in the diagonal of the matrix in the result. The Dodgson polynomials on the off-diagonals require more care. We orient the edges like this: edge 2 goes towards vertex 1 and the other two away from it.

\[
  \includegraphics{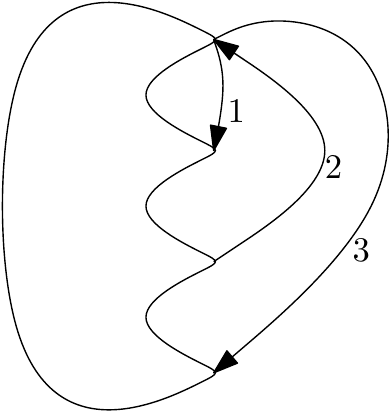}
\]
This ensures all the off-diagonal signs are negative (by Lemma \ref{Lemma signs}) and that each Dodgson polynomial gives the desired spanning forest polynomial. The result follows.
\end{proof}

Note that the matrix in Lemma \ref{apply Jacobi} is the Laplacian matrix with row and column 1 removed for the following graph
\[
\includegraphics[scale=0.90]{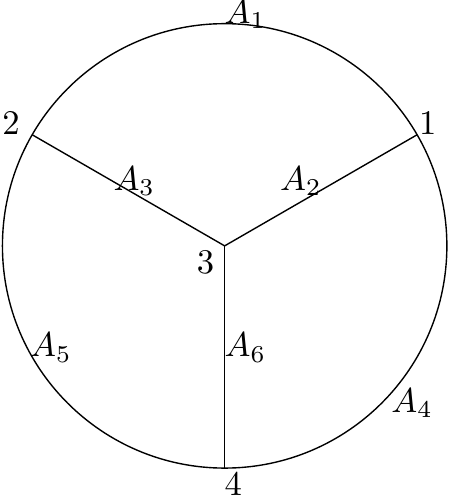}
\]
where the edge labels are the $A$'s. This is not a coincidence and there is a general identity which we leave out for brevity. However, the statement is analogous.

To complete the calculation we need to multiply the whole expression by $P$ and use the following
\begin{lemma}\label{PAi}
\begin{align*}
  P\Aa & = \Ba\Bb + \Ba\Be + \Bb\Be + \Be\Bf + \Be\Bg - \Bf\Bg  \\ 
  P\Ab & = \Ba\Bc + \Ba\Bf + \Bc\Bf + \Be\Bf - \Be\Bg + \Bf\Bg  \\ 
  P\Ac & = \Ba\Bd + \Ba\Bg + \Bd\Bg - \Be\Bf + \Be\Bg + \Bf\Bg  \\ 
  P\Ad & = \Bb\Bc + \Bb\Bg + \Bc\Bg - \Be\Bf + \Be\Bg + \Bf\Bg  \\ 
  P\Ae & = \Bb\Bd + \Bb\Bf + \Bd\Bf + \Be\Bf - \Be\Bg + \Bf\Bg  \\ 
  P\Af & = \Bc\Bd + \Bc\Be + \Bd\Be + \Be\Bf + \Be\Bg - \Bf\Bg 
\end{align*}
\end{lemma}

\begin{proof}
 By symmetry of the four vertices it suffices to prove the formula for $P\Aa$. Consider the graph
  \[
  \includegraphics{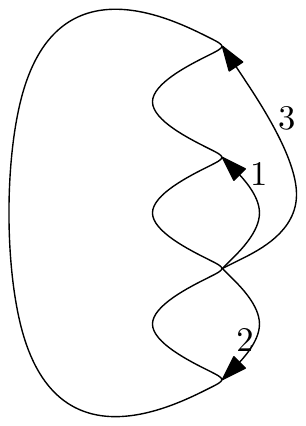}
  \]
  Then
  \begin{align*}
    P\Aa & = -\Psi^{123,123}\Psi^{1,3}_2 \quad \text{by Lemma \ref{Lemma signs}}\\
    & = \Psi^{12,32}\Psi^{13,13}_2 - \Psi^{12,31}\Psi^{13,23} \quad \text{by \eqref{klevelmod} with $A=\{1\}$, $B=\{3\}$, $I=\{2,3\}$,}\\
    & \phantom{\Psi^{12,32}\Psi^{13,13}_2 - \Psi^{12,31}\Psi^{13,23} \quad } \qquad  \text{$J=\{1,2\}$, and $E=\{1,2,3\}$} \\
    & = (1,1,2,-)(-,-,1,2) - (1,-,2,1)(-,1,2,1) \quad \text{by Lemma \ref{Lemma signs}} \\
    & = (\Bb+\Be)(\Ba+\Bb+\Bf+\Bg) - (\Bb+\Bg)(\Bb+\Bf) \\
    & = \Ba\Bb + \Ba\Be + \Bb\Be + \Be\Bf + \Be\Bg - \Bf\Bg
  \end{align*}
\end{proof}

Now we find out what happens when we multiply the equation in Lemma \ref{apply Jacobi} by $P$.
\begin{lemma}\label{main calc}
\[
((1,1,1,1)(1,2,3,4))^2 = E(0,0,0,0,0,0,0,0)E(0,1,0,1,1,1,1,1)
\]
where $E(x_1, x_2, x_3, x_4, x_5, x_6, x_7, x_8)$ is the right hand side of \eqref{emp formula}.
\end{lemma}

\begin{proof} By definition
\begin{equation}\label{emp1}
E(0,0,0,0,0,0,0,0) = (A_5+A_6)(B_1+B_7) + A_6(B_2+B_6) + A_4(B_1+B_6)
\end{equation}
and 
\begin{equation}\label{emp2}
E(0,1,0,1,1,1,1,1) = (A_1+A_2)(B_4+B_7) + A_2(B_2+B_5) + A_4(B_4+B_5).
\end{equation}

\allowdisplaybreaks

Use Lemma \ref{apply Jacobi} and \ref{PAi} to calculate $((1,1,1,1)(1,2,3,4))^2$.  With some trial and error we chose which lines of Lemma \ref{PAi} to use so that the final result would look as much as possible like the product of \eqref{emp1} and \eqref{emp2}. The term $\left((1,1,1,1)(1,2,3,4)\right)^2$ equals
\begin{align*}
 &\textcolor{blue}{(A_1+A_2)(PA_3)(A_5+A_6)} + \textcolor{darkcyan}{A_2(PA_1)(A_5+A_6)} + \textcolor{darkgreen}{(A_1+A_2)(PA_5)A_6} \\
& \qquad + \textcolor{red}{A_4(A_1+A_2+A_5+A_6)(PA_3)} + \textcolor{magenta}{A_4(A_2+A_6)(PA_1+PA_5)} \\
& = \textcolor{blue}{(A_1+A_2)(A_5+A_6)(B_1B_7 + B_4B_7 + B_1B_4)} \\
& \qquad + \textcolor{darkcyan}{A_2(A_5+A_6)(B_1B_2+B_2B_5+B_1B_5+B_5B_7)} \\
& \qquad + \textcolor{darkgreen}{(A_1+A_2)A_6(B_2B_6+B_4B_6+B_2B_4+B_6B_7)} \\
& \qquad + \textcolor{red}{A_4(A_1+A_2+A_5+A_6)(B_1B_7 + B_4B_7 + B_1B_4 + B_5B_7+B_6B_7)} \\
& \qquad + \textcolor{magenta}{A_4(A_2+A_6)(B_1B_5 + B_2B_5 + B_1B_2 + B_2B_6 + B_4B_6 + B_2B_4+B_5B_6)}\\
& \qquad + B_5B_6A_2A_6 - \textcolor{blue}{B_5B_6A_1A_5} - \textcolor{red}{B_5B_6A_4(A_1+A_5)} \\
& \qquad + B_5B_7(A_1+A_2)A_5 + B_6B_7A_1(A_5+A_6)
\end{align*}

Now we consider the difference between this expression and \eqref{emp1} times \eqref{emp2}
\begin{align*}
& \textcolor{blue}{A_1A_5(-B_7^2 - B_5B_6 + B_5B_7 + B_6B_7) + A_1A_6(-B_7^2 + B_2B_6 - B_2B_7 + B_6B_7) }\\
  & \qquad \textcolor{blue}{+ A_2A_5(-B_7^2 + B_2B_5 - B_2B_7 + B_5B_7) + A_2A_6(-B_7^2 - 2B_2B_7 -B_2^2) }\\
  & \qquad + A_4A_5\textcolor{darkcyan}{(B_1B_7 + B_6B_7 - B_1B_5 - B_5B_6)} + A_4A_6\textcolor{darkgreen}{(B_1B_7 + B_6B_7 + B_1B_2 + B_2B_6)} \\
  & \qquad + A_1A_4\textcolor{red}{(B_4B_7 + B_5B_7 - B_6B_4 - B_5B_6)} + A_2A_4\textcolor{magenta}{(B_4B_7 + B_5B_7 + B_2B_5 + B_2B_4)} \\
  & \qquad - A_4^2(B_1+B_6)(B_4+B_5) \\
  & = \textcolor{blue}{-\big(A_6(B_2+B_7)- A_5(B_5-B_7)\big)\big(A_2(B_2+B_7)-A_1(B_6-B_7)\big)} \\
  & \qquad - A_4A_5\textcolor{darkcyan}{(B_1+B_6)(B_5-B_7)} + A_4A_6\textcolor{darkgreen}{(B_1+B_6)(B_2+B_7)} \\
  & \qquad - A_1A_4\textcolor{red}{(B_4+B_5)(B_6-B_7)} + A_2A_4\textcolor{magenta}{(B_2+B_7)(B_5+B_4)} - A_4^2(B_1+B_6)(B_4+B_5) \\
  & = -A_4(B_4+B_5)\big(A_2(B_2+B_7)-A_1(B_6-B_7)\big) \qquad \text{ by \eqref{formain}}\\
  & \qquad \textcolor{purple}{- A_4A_5(B_1+B_6)(B_5-B_7) + A_4A_6(B_1+B_6)(B_2+B_7)} \\
  & \qquad - A_1A_4(B_4+B_5)(B_6-B_7) + A_2A_4(B_2+B_7)(B_5+B_4) \textcolor{purple}{- A_4^2(B_1+B_6)(B_4+B_5)} \\
  & = \textcolor{purple}{A_4(B_1+B_6)\big(A_6(B_2+B_7) + A_5(B_7-B_5) - A_4(B_4+B_5)\big)} \\
  & = 0 \qquad \text{ by \eqref{formain}}
\end{align*}
\end{proof}

We are now ready to finish the proof of the main theorem.
\begin{proof}[Proof of Theorem \ref{main thm}]
By Lemma \ref{main calc} we know that
\[
E(0,0,0,0,0,0,0,0)E(0,1,0,1,1,1,1,1) = \big((1,2,3,4)(1,1,1,1)\big)^2
\]
and by Proposition \ref{free vars} we know that $E$ does not depend on the $x_i$.  Thus we have
\[
E(x_1, x_2, x_3, x_4, x_5, x_6, x_7, x_8) = \pm (1,2,3,4)(1,1,1,1)
\]
It remains to check the sign. Note that $(1,1,1,1)=P=\Psi_G$ and $(1,2,3,4)$ is $\Psi$ for $G$ with $v_1$, $v_2$, $v_3$, and $v_4$ identified. Since both $(1,1,1,1)$ and $(1,2,3,4)$ are Kirchhoff polynomials of graphs all monomials appear with nonnegative coefficients.  Looking at $\eqref{emp1}$ we see that the sign is $1$ and the proof is complete.
\end{proof}

\section{Conclusions}\label{sec concl}

Theorem \ref{main thm} gives a nice generalization of \eqref{3 and 1 id}.  Equation \eqref{3 and 1 id} itself is crucial to the combinatorial and algebro-geometric approach to understanding the periods of Feynman integrals \cite{Sphi4, bek, BrY, BrS, AMdet}.  In such work, having a good intuition of how to massage the polynomials which occur is crucial, and it is the second author's experience that spanning forest polynomials and their identities are very useful in this regard.

We can ask for an edge-transferring interpretation of Theorem \ref{main thm}, comparable to what we discussed for \eqref{3 and 1 id} in the introduction.  Consider \eqref{emp1}, which is the result of setting the free variables to $0$ in our main theorem.  Collecting terms this gives
\[
  (1,2,3,4)(1,1,1,1) = (1,2,3,1)(1,-,1,2) + (1,2,3,2)(-,1,1,2) + (1,2,3,3)(-,-,1,2)
\]
which says that we can choose to transfer an edge from any spanning forests contributing to $(1,1,1,1)$ to one of those contributing to $(1,2,3,4)$, so that we always merge the tree of the last vertex from $(1,2,3,4)$ into one of the other trees, and always split the last and second last vertices of $(1,1,1,1)$ into separate trees.  Furthermore, the identity describes precisely how the split trees will interact with the remaining vertices.  We know of no direct combinatorial proof which follows this interpretation.

We initially obtained \eqref{emp formula} by a numerical calculation.  We first picked a graph on which to perform the calculations -- we picked $K_4$, $K_5$ and $K_6$. Then we calculated each $\Ai$ and $\Bi$ on this graph and then formed all possible products of $A$'s and $B$'s and formed the sum $\sum_{s,t}x_{st}A_s B_t$, where $x_{st}$ is a constant, $1\leq s \leq 6$ and $1\leq t\leq 7$ for a total of 42 constants, and solved the linear system. 
The initial numerical calculation could, a priori, have had spurious degrees of freedom, but it could not miss any true identity of the desired form.  Consequently, \eqref{emp formula} is the most general quadratic formula involving 4 marked vertices.

A natural questions is what do formulae for more marked vertices look like. Numerical calculations show that for 5 and 6 marked vertices the formulae have 15 and 24 free variables. For the classical Dodgson identity, the $A$'s and $B$'s are the same.  If we treat the $A$'s and the $B$'s as different, we have a formula with 3 free variables. Trivially, a formula for 2 marked vertices has no free variables. For $n = 2,3,4,5$ and $6$ the identities so far point to expressions having 0, 3, 8, 15 and 24 variables in formulae for $n$ marked vertices. These numbers are generated by $n(n-2)$ for $n= 2,3,4,5$ and $6$.

\bibliographystyle{plain}
\bibliography{main}

\begin{thebibliography}{10}

\bibitem{AMdet}
Paolo Aluffi and Matilde Marcolli.
\newblock Parametric feynman integrals and determinant hypersurfaces.
\newblock {\em Adv. Theor. Math. Phys.}, 14(3):911--964, 2010.
\newblock arXiv:0901.2107.

\bibitem{bek}
Spencer Bloch, H\'el\`ene Esnault, and Dirk Kreimer.
\newblock On motives associated to graph polynomials.
\newblock {\em Commun. Math. Phys.}, 267:181--225, 2006.
\newblock arXiv:math/0510011v1 [math.AG].

\bibitem{BW}
Christian Bogner and Stefan Weinzierl.
\newblock Feynman graph polynomials.
\newblock arXiv:1002.3458.

\bibitem{Brbig}
Francis Brown.
\newblock On the periods of some {F}eynman integrals.
\newblock arXiv:0910.0114.

\bibitem{Brphi4}
Francis Brown.
\newblock The massless higher-loop two-point function.
\newblock {\em Commun. Math. Phys}, 287:925--958, 2009.
\newblock arXiv:0804.1660.

\bibitem{BrS}
Francis Brown and Oliver Schnetz.
\newblock A {K3} in $\phi^4$.
\newblock arXiv:1006.4064.

\bibitem{BrY}
Francis Brown and Karen Yeats.
\newblock Spanning forest polynomials and the transcendental weight of
  {F}eynman graphs.
\newblock {\em Commun. Math. Phys.}, 301(2):357--382, 2011.
\newblock arXiv:0910.5429.

\bibitem{Chai}
Seth Chaiken.
\newblock A combinatorial proof of the all minors matrix tree theorem.
\newblock {\em SIAM J. Alg. Disc. Meth.}, 3(3):319--329, 1982.

\bibitem{dodgson}
C.~L. Dodgson.
\newblock Condensation of determinants, being a new and brief method for
  computing their arithmetic values.
\newblock {\em Proc. Roy. Soc. Ser. A}, 15:150--155, 1866.

\bibitem{Kcore}
Dirk Kreimer.
\newblock The core {H}opf algebra.
\newblock In {\em Quanta of Maths}, volume~11 of {\em Clay Mathematics
  Proceedings}, pages 313--321, 2010.
\newblock arXiv:0902.1223.

\bibitem{Sphi4}
Oliver Schnetz.
\newblock Quantum periods: A census of $\phi^4$-transcendentals.
\newblock {\em Communications in Number Theory and Physics}, 4(1):1--48, 2010.
\newblock arXiv:0801.2856.

\bibitem{Zeil}
Doron Zeilberger.
\newblock Dodgson's determinant-evaluation rule proved by two-timing men and
  women.
\newblock {\em Elec. J. Combin.}, 4(2), 1997.

\end{thebibliography}

\end{document}